\documentclass[11pt]{amsart}

\usepackage{amsfonts}
\usepackage{amscd}
\usepackage{amsbsy}
\usepackage{amsmath}
\usepackage{amsthm}
\usepackage{setspace}
\usepackage{url}

\theoremstyle{plain}
\newtheorem{theorem}{Theorem}[section]
\newtheorem{lemma}[theorem]{Lemma}
\newtheorem{proposition}[theorem]{Proposition}

\theoremstyle{definition}
\newtheorem{remark}[theorem]{Remark}

\newtheorem{example}[theorem]{Example}

\newcommand{\MM}{\mathcal M}
\newcommand{\BM}{\overline{\mathcal M}}

\newcommand{\PP}{\mathbb P}

\newcommand{\BHH}{\overline{\mathcal H}}
\newcommand{\CC}{\mathcal C}

\newcommand{\OO}{\mathcal O}

\newcommand{\RR}{\mathbb R}
\newcommand{\HH}{\mathcal H}

\newcommand{\TT}{\mathcal T}

\newcommand{\BEff}{\overline{\operatorname{Eff}}}

\newcommand{\other}{\operatorname{other}}

\newcommand{\Pic}{\operatorname{Pic}}
\newcommand{\bbC}{\mathbb C}
\newcommand{\bbE}{\mathbb E}
\newcommand{\bbQ}{\mathbb Q}

\newcommand{\bbF}{\mathbb F}
\newcommand{\rank}{\operatorname{rank}}

\newcommand{\SL}{\operatorname{SL}}

\title[abelian differentials and the Teichm\"uller dynamics]{Strata of abelian differentials and the Teichm\"uller dynamics}

\begin{document}
\bibliographystyle{halpha}

\author{Dawei Chen}

\address{Department of Mathematics, Boston College, Chestnut Hill, MA 02467, USA}

\email{dawei.chen@bc.edu}

\thanks{During the preparation of this work the author is partially supported by NSF grant DMS-1200329. }

\begin{abstract}
This paper focuses on the interplay between the intersection theory and the Teichm\"uller dynamics on the moduli space of curves. As applications, 
we study the cycle class of strata of the Hodge bundle, present an algebraic method to calculate the class of the divisor parameterizing abelian differentials with a non-simple zero, and verify a number of extremal effective divisors on the moduli space of pointed curves in low genus. 
\end{abstract}

\maketitle

\section{Introduction}

Let $\HH$ be the Hodge bundle parameterizing $(C, \omega)$, where $C$ is a smooth curve of genus $g$ and $\omega$ is an abelian differential on $C$, i.e. a section of the canonical line bundle. On the moduli space $\MM_g$ of genus $g$ curves, $\HH$ is a vector bundle of rank $g$. Let $\mu = (m_1, \ldots, m_n)$ be a partition of $2g-2$. Denote by $\HH(\mu)$ the locus of pairs $(C, \omega)$ in $\HH$ such that $(\omega)_0 = \sum_{i=1}^n m_i p_i$ for distinct points $p_1, \ldots, p_n$ in $C$. It is well-known that $\HH(\mu)$ is a submanifold of $\HH$ whose local coordinates are given by the relative periods, see \cite{kontsevich}.  

There is an $\SL_2(\RR)$-action on $\HH$, called the Teichm\"uller dynamics, that changes the flat structure of $C$ induced by $\omega$. A central question in the study of Teichm\"uller dynamics is to understand the structure of its orbit closures. What are their dimensions? Do they possess a manifold structure? What are the associated dynamical quantities, such as Lyapunov exponents and Siegel-Veech constants? We refer to \cite{emz, kz03, ekz} for a comprehensive introduction to these subjects. 

Although the questions are analytic in nature, recently there have been some attempts using tools in algebraic geometry to study them. For instance, if the projection of an orbit forms an algebraic curve in $\MM_g$, we call it a Teichm\"uller curve. A decade ago Kontsevich and Zorich made a conjecture about the 
non-varying phenomenon of sums of Lyapunov exponents for Teichm\"uller curves in low genus. Marking the zeros of an abelian differential, one can lift a Teichm\"uller curve to the Deligne-Mumford moduli space $\BM_{g,n}$ of stable genus $g$ curves with $n$ marked points. In \cite{chenmoeller} the conjecture was proved by calculating the intersection of Teichm\"uller curves with divisor classes on $\BM_{g,n}$. Furthermore in \cite{chenmoellerzagier} the authors consider a special type of higher dimensional orbit closures given by torus coverings. The upshot also relies on certain intersection calculation on the Hurwitz space compactified by admissible covers 
\cite[Chapter 3.G]{harrismorrison}. 

A complete classification of the orbit closures is still missing. Nevertheless, the $\SL_2(\RR)$-action preserves every stratum $\HH(\mu)$, hence it is already interesting to study these strata from the viewpoint of intersection theory. To set it up algebraically, we projectivize $\HH$ by modulo $\bbC^{*}$. In other words, the projectivization $\PP \HH$ parameterizes canonical divisors instead of abelian differentials. As before if we mark the zeros of a differential, one can lift a stratum $\PP\HH(\mu)$ to $\MM_{g,n}$. 
Then the first step of the framework is to understand its cycle class in the Chow ring of $\MM_{g,n}$. In Section~\ref{sec:strataclass} we answer this question as a consequence of the Porteous formula. 

Note that $\PP\HH$ extends to the boundary of $\BM_g$ as a projective bundle $\PP\BHH$. Alternatively, one can take the closure of a stratum in $\PP\BHH$ and study its cycle class. A precise description for the boundary of $\PP\BHH(\mu)$ is still unknown in general, which causes a technical obstruction. However for codimension-one degeneration, i.e. for $\mu = (2, 1^{2g-4})$, the stratum $\PP\BHH(2, 1^{2g-4})$ is a divisor in $\PP\BHH$. In Section~\ref{sec:divisorclass} we are able to calculate its class (Proposition~\ref{prop:classinterior} and Theorem~\ref{thm:class}). Understanding this divisor is useful from a number of aspects, e.g. in \cite{hamenstadt} it was used to detect signatures of surface bundles. We remark that its divisor class was first calculated in \cite[Theorem 2]{kztau} by an analytic approach using the Tau function. Our method is purely algebraic, hence it provides additional information regarding the birational geometry of $\PP\BHH$. For instance, it is often useful but difficult to find an extremal effective divisor intersecting the interior of a moduli space. The existence of such a divisor can provide crucial information for the birational type of the moduli space. As a by-product of our intersection theoretical approach, we show that $\PP\BHH(2, 1^{2g-4})$ lies on the boundary of the pseudo-effective cone of $\PP\BHH$ (Proposition~\ref{prop:rigid}). 

In Section~\ref{sec:extremal} we reverse our engine, using the Teichm\"uller dynamics to study effective divisors on the moduli space of curves. 
The history of studying effective divisors on $\BM_g$ dates back to \cite{harrismumford}, where Harris and Mumford used the Brill-Noether divisor parameterizing 
curves with exceptional linear series to show that $\BM_g$ is of general type for large $g$. Later on Logan studied a series of pointed Brill-Noether divisors on $\BM_{g,n}$ \cite{Logan}. As mentioned above, we would like to understand whether those divisors are extremal. By checking their intersections with various 
Teichm\"uller curves, we prove the extremality for a number of pointed Brill-Noether divisors in low genus (Theorems~\ref{thm:n=1} -- \ref{thm:n=5}). 

{\bf Acknowledgments:} The author thanks Martin M\"oller for many stimulating discussions about the Teichm\"uller dynamics, thanks Peter Zograf for clarifying a sign issue in the divisor class of $\PP\BHH(2, 1^{2g-4})$, and thanks Izzet Coskun and David Jensen for helpful comments on the extremal divisors in Section~\ref{sec:extremal}. Part of the work was done when the author visited the Mathematical Sciences Center at Tsinghua University in Summer 2012. The author thanks their invitation and hospitality. 

\section{Classes of the strata}
\label{sec:strataclass}

The calculation in this section is standard to an algebraic geometer. But we still write down everything in detail for the readers who are only familiar with the dynamical side of the story. Let us first consider a more general setting. 

Denote by $\mu = (m_1, \ldots, m_n)$ a partition of a positive integer $d$. Let $BN^r_{\mu}$ be the locus of $(C, p_1, \ldots, p_n)$ in $\MM_{g,n}$ such that 
the divisor $D = \sum_{i=1}^n m_i p_i$ in $C$ satisfies $$h^0(C, D) \geq r+1. $$ 
Here the notation $BN$ stands for the Brill-Noether divisors, see \cite[Chapter V]{ACGH} for a comprehensive introduction to the Brill-Noether theory. Below we calculate the class of $BN^r_{\mu}$ in the Chow ring of $\MM_{g,n}$. 

Let $\CC^{n}$ denote the $n$-fold fiber product of the universal curve $\CC$ over $\MM_g$. Let 
$$\pi: \CC^{n+1} \to \CC^n$$ 
be the projection forgetting the last factor. Let $\Omega_i$ be the pullback of the dualizing sheaf via the $i$th projection and denote its first Chern class by $\omega_i$. Define 
$$\Delta_{\mu} = \sum_{i=1}^n m_i \Delta_{i, n+1}, $$ 
where $\Delta_{i,j}$ is the diagonal corresponding to $p_i = p_j$. 

Define a sheaf $\bbF_{\mu}$ on $\CC^{n}$ as 
$$ \bbF_{\mu} = \pi_{*}(\OO_{\Delta_{\mu}} \otimes \Omega_{n+1}). $$
The stalk of $\bbF_{\mu}$ at a point $(C, p_1, \ldots, p_n)$ can be identified with 
$$ H^0(C, K/K(- D)) \cong \mathbb C^{d}, $$
where $K$ is the canonical line bundle of $C$. Consequently $\bbF_{\mu}$ is a vector bundle of rank $d$. 
Let $\bbE$ be the pullback of the Hodge bundle on $\CC^{n}$, i.e. its fiber at $(C, p_1, \ldots, p_n)$ is canonically given by $H^0(C, K)$. We have an evaluation map 
$$ \phi: \bbE \to \bbF_{\mu}. $$ 
The locus where 
$$\rank(\phi) \leq d - r$$
parameterizes $(C, p_1, \ldots, p_n)$ satisfying
$$h^0(C, K(-D)) \geq g - d + r, $$ 
namely, $h^0(C, D) \geq r + 1$ by the Riemann-Roch formula. 

Recall the Porteous formula 
$$\Delta_{p,q} \Big(\sum_{i=0}^{\infty} c_i t_i\Big) = \det \left( \begin{array}{cccc} 
c_p & c_{p+1} & \cdots & c_{p+q-1} \\
c_{p-1} & c_p & \cdots & c_{p+q - 2} \\
\vdots & \vdots &  \ddots   & \vdots \\
c_{p-q+1} & c_{p-q+2} & \cdots & c_{p} 
\end{array}\right), $$
see \cite[Chapter II $\S4$ (iii)]{ACGH}. 
We conclude that the class of $BN^r_{\mu}$ in the Chow ring of $\MM_{g,n}$ 
is given by 
$$\Delta_{r, g- d + r} (c(\bbF_{\mu})/ c(\bbE)). $$

Let $\mu' = (m_1 - 1, m_2, \ldots, m_n)$ and similarly define $\bbF_{\mu'}$.  
We have a filtration 
$$ 0 \to \bbF_1 \to \bbF_{\mu} \to \bbF_{\mu'} \to 0 $$
where the fiber of $\bbF_1$ at $(C, p_1, \ldots, p_n)$ is 
$$ H^0(C, K(-D + p_1)/K(-D)), $$
namely, $\bbF_1$ is isomorphic to $\Omega_1^{\otimes m_1}$. 
Then by induction we have 
\begin{eqnarray*}
 c(\bbF_{\mu}) & = & \prod_{i=1}^n \prod_{j=1}^{m_i} (1 + j\omega_i) \\     
                          & = & 1 + \Big(\sum_{i=1}^n \frac{m_i(m_i+1)}{2} \omega_i\Big) + \cdots 
\end{eqnarray*}

Let $\lambda_i = c_i (\bbE)$ the $i$th Chern class of the Hodge bundle. 
We have 
\begin{eqnarray*}
\frac{1}{c(\bbE)} & = & 1 - \lambda_1 + (\lambda^2_1 - \lambda_2)  + \cdots  \\
                             & = & 1 - \lambda_1 + \frac{1}{2}\lambda_1^2 + \cdots
\end{eqnarray*}
where we use the relation $\lambda_2 = \lambda_1^2 / 2$. Indeed all the $\lambda_i$ can be expressed as polynomials of the Mumford class
$\kappa_1 = 12\lambda_1$ and these polynomials can be worked out explicitly in any given case, see \cite[Chapter 3.E]{harrismorrison}. 

\begin{example}
Let $\mu = (m_1, \ldots, m_n)$ be a partition of $g$. 
Consider the divisor $BN^1_{\mu}$ in $\MM_{g,n}$. In this case $r = 1$ and $d = g$, 
hence we conclude that the class of $BN^1_{\mu}$ in $\MM_{g,n}$ is 
$$c_1(\bbF) - c_1(\bbE) = - \lambda_1 + \sum_{i=1}^{n} \frac{m_i(m_i+1)}{2} \omega_i. $$
The class of the closure of $BN^1_{\mu}$ in $\BM_{g,n}$ was calculated in \cite[Theorem 5.4]{Logan}. As we see the results are the same
modulo boundary classes. 
\end{example}

Now we specialize to the strata of abelian differentials. Let $\mu = (m_1, \ldots, m_n)$ be a partition of $2g-2$. 
Marking the zeros of a differential, we can embed $\PP\HH(\mu)$ into $\MM_{g,n}$. In order to calculate its class, in the Porteous formula we have $r = g-1$ and $d = 2g-2$, hence the class of $\PP\HH(\mu)$ is given by 
$$ \Delta_{g-1,1}(c(\bbF_{\mu})/ c(\bbE)) = [c(\bbF_{\mu})/ c(\bbE)]_{g-1}. $$
More precisely, we have 
$$ c(\bbF_{\mu}) = \prod_{i=1}^{n} \prod_{j=1}^{m_i} (1 + j\omega_i), $$
$$ \frac{c(\bbF_{\mu})}{c(\bbE)} = \Big(\prod_{i=1}^{n} \prod_{j=1}^{m_i} (1 + j\omega_i)\Big)\cdot \Big(1 - \lambda_1 + \frac{1}{2}\lambda_1^2 + \cdots\Big). $$ 
Then the term of degree $g-1$ determines the desired class. 

\section{Abelian differentials with a non-simple zero}
\label{sec:divisorclass}

In this section we consider the locus $\PP\HH(2, 1^{2g-4})$ parameterizing canonical divisors with a zero of multiplicity $\geq 2$. 
Note that it forms a divisor in $\PP\HH$. The rational Picard group of $\PP\HH$ is generated by $\lambda$ and $\psi$, where $\lambda$ is the pullback of the Hodge class $\lambda_1$ from $\MM_g$ and $\psi$ is the class of the universal line bundle $\OO_{\PP\HH}(1)$. Here the projectivization $\PP V$ of a vector space $V$ parameterizes lines instead of hyperplanes. As a result, the divisor classes $\lambda$, $\psi$ and etc in our setting are the opposites of the corresponding classes in \cite{kztau}. 

\begin{proposition}
\label{prop:classinterior}
The divisor class of $\PP\HH(2, 1^{2g-4})$ in $\Pic_{\bbQ}(\PP\HH)$ is given by 
$$ \PP\HH(2, 1^{2g-4}) = (6g-6) \psi - 24\lambda. $$
\end{proposition}

\begin{proof}
Suppose the class is 
$$\PP\HH(2, 1^{2g-4}) = a\psi + b\lambda. $$  
Take a general curve $C$ of genus $g$ and consider its canonical embedding in $\PP H^0(C, K)\cong\PP^{g-1}$. A pencil $B$ of canonical divisors in $\PP H^0(C, K)$ corresponds to hyperplanes in $\PP^{g-1}$ containing a fixed linear subspace $\Lambda = \PP^{g-3}$. Project $C$ from $\Lambda$ to a line, which induces a 
map $C\to \PP^1$ of degree $2g-2$. Note that the number of simple ramification points of the map is equal to the intersection number $B\cdot  \PP\HH(2,1^{2g-4})$. Moreover, we have 
$$ B\cdot \psi =1, \quad B\cdot \lambda = 0. $$
By the Riemann-Hurwitz formula, we conclude that 
$$ a = 2g-2 + 2 (2g-2) = 6g-6. $$

In order to calculate $b$, we use the relation 
$$ \lambda = \kappa_{\mu}\cdot \psi $$ restricted to a stratum 
$\PP\HH(\mu)$, see \cite[Section 3.4]{ekz} and \cite[Section 4]{chenmoeller}, where 
for a partition $\mu = (m_1, \ldots, m_n)$, $\kappa_{\mu}$ is defined as  
$$\kappa_{\mu} = \frac{1}{12}\sum_{i=1}^{n} \frac{m_i (m_i+2)}{m_i + 1}.$$ 
In particular, we have $$\kappa_{(1^{2g-2})} = \frac{g-1}{4}.$$ 
Since the complement of the principal stratum $\PP\HH(1^{2g-2})$ in $\PP\HH$ consists of the divisorial stratum $\PP\HH(2, 1^{2g-4})$ union strata of higher codimension, 
in $\Pic_{\bbQ}(\PP\HH)$ we conclude that 
$$ \lambda = \frac{g-1}{4} \psi + c \cdot \PP\HH(2,1^{2g-4})$$ 
with $c$ unknown. 

Using the test curve $B$ again, we have 
$$ \frac{g-1}{4} + (6g-6) c = 0, $$
$$ c = - \frac{1}{24}. $$
Therefore, we conclude that 
$$ \PP\HH(2,1^{2g-4}) = - 24 \lambda + (6g-6) \psi. $$
\end{proof}

Next we take the closure $\PP\BHH(2, 1^{2g-4})$ in the projective bundle $\PP\BHH$ over $\BM_{g}$. Still use $\delta_i$ to denote the pullback of the boundary divisor 
$\delta_i$ from $\BM_g$. Then the rational Picard group of $\PP\BHH$ is generated by $\lambda$, $\psi$ and $\delta_0, \ldots, \delta_{[g/2]}$. Now we can calculate the full class of $\PP\BHH(2, 1^{2g-4})$ including the boundary divisors. 

\begin{theorem}
\label{thm:class}
In $\Pic_{\bbQ}(\PP\BHH)$, we have 
$$ \PP\BHH(2,1^{2g-4}) = (6g-6)\psi - 24\lambda + 2\delta_0 + 3 \sum_{i=1}^{[g/2]}\delta_i. $$
\end{theorem}

\begin{proof}
In the proof of Proposition~\ref{prop:classinterior}, the relation 
$ \lambda = \kappa_{\mu} \cdot \psi $
restricted to a stratum $\PP\HH(\mu)$ arises from   
Noether's formula $12 \lambda = \kappa_1 + \delta$ modulo boundary, where 
$\kappa_1$ is the Mumford class and $\delta$ is the total boundary class. 
Note that a general degeneration from the principal stratum $\PP\HH(1^{2g-2})$ to the boundary $\delta_0$ keep the $2g-2$ sections of simple zeros away from the non-separating node. Then we can extend Noether's formula as 
$12 \lambda = \kappa_1 + \delta_0$ restricted to $\PP\HH(1^{2g-2})$ union $\delta_0$. By Proposition~\ref{prop:classinterior} we can rewrite this relation as 
$$ \PP\BHH(2,1^{2g-4}) = (6g-6) \psi - 24\lambda + 2\delta_{0} $$
in $\Pic_{\bbQ}(\PP\BHH)$ modulo $\delta_i$ for $i > 0$. 

Next, take a general one-dimensional family $B$ of genus $g$ curves with $2g-2$ sections such that in a generic fiber the sum of the sections yields a canonical divisor. 
Moreover, suppose there are $k$ special fibers $C$ that consist of two components $C_1$ and $C_2$ joined at a separating node $t$. Then $t$ has to be a zero of the canonical divisor of $C$ restricted to $C_1$ and $C_2$. In other words, two of the $2g-2$ sections meet at $t$. Without loss of generality we can assume that the first two sections meet at $t$, the next $2i-2$ sections meet $C_1$ and the last $2(g-i)-2$ sections meet $C_2$, where $i$ is the genus of $C_1$ for $1\leq i \leq [g/2]$. 

Blow up the family at these $k$ nodes, i.e. we insert a rational bridge 
$E$ between $C_1$ and $C_2$ with $E^2 = -2$. Let $\pi: \CC\to B$ denote the resulting family, $S_1, \ldots, S_{2g-2}$ the proper transforms of the sections, and 
$E_1, \ldots, E_k$ the $k$ exceptional curves. As an analogue of the exact sequence in \cite[Proof of Proposition 4.8]{chenmoeller}, we have 
$$ 0 \to \pi^{*}\OO(1) \to \Omega_{\pi}\otimes \OO_{\CC}\Big(-\sum_{i=1}^k E_i\Big)\to \sum_{i = 1}^{2g-2} \OO_{S_i}(S_i) \to 0. $$
Here $\OO(1)$ is the universal line bundle whose first Chern class is $\psi$ and $\Omega_{\pi}$ is the relative dualizing sheaf associated to $\pi$. 
The middle term restricted to $C_i$ is the canonical line bundle of $C_i$ and it is trivial restricted to $E_i$.  
Let $\omega = c_1(\Omega_\pi)$. 
Then we conclude that 
$$ \omega = \pi^{*}\psi + \sum_{i=1}^k E_i + \sum_{i=1}^{2g-2} S_i. $$
Moreover, we have 
$$ E_i\cdot S_1 = E_i\cdot S_2 = 1, $$
$$ E_i \cdot S_j = 0, \ j > 2, $$
$$ E_i^2 = -2, \ E_i \cdot E_j = 0, \ i\neq j. $$
Then for $i\neq 1, 2$ we see that 
$$ \pi_{*}(S_i^2) = - \pi_{*}(\omega\cdot S_i) = - \psi -  \pi_{*}(S_i^2),$$  
$$ \pi_{*}(S_i^2) =  -\frac{1}{2}\psi. $$
Similarly we have 
$$ \pi_{*}(S_1^2) = \pi_{*}(S_2^2) = -\frac{1}{2}\psi - \frac{k}{2}. $$
Therefore, we obtain that 
\begin{eqnarray*}
 \pi_{*}(\omega^2) & = & 2 (2g-2) \psi - 2k + 4k - (g-1)\psi - k \\
                                                     & = & (3g-3)\psi + k. 
\end{eqnarray*} 
Since $E_i^2 = -2$ and the family $B$ intersects $\delta_i$ with multiplicity two, i.e. $B\cdot \delta_{i} = 2k$, 
the above calculation implies that  
$$ \pi_{*}(\omega^2) = (3g-3)\psi + \frac{1}{2}\sum_{i=1}^{[g/2]} \delta_{i}. $$

By Noether's formula, restricted to $B$ we have 
\begin{eqnarray*}
\lambda & = &\frac{\pi_{*}(\omega^2) + \delta}{12} \\
                & = & \frac{1}{12}\Big((3g-3)\psi + \frac{1}{2}\sum_{i=1}^{[g/2]} \delta_{i} + \sum_{i=0}^{[g/2]}\delta_i\Big)\\
                & = & \frac{g-1}{4}\psi + \frac{1}{12}\delta_0 + \frac{1}{8}\sum_{i=1}^{[g/2]}\delta_i.                                                                                                     
\end{eqnarray*}
In other words, in $\Pic_{\bbQ}(\PP\BHH)$ we have 
$$ \lambda = \frac{g-1}{4}\psi + \frac{1}{12}\delta_0 + \frac{1}{8}\sum_{i=1}^{[g/2]}\delta_i + c \cdot \PP\BHH(2,1^{2g-4}). $$
But we have seen that $c = -1/24$. Therefore, we thus conclude the desired divisor class. 
\end{proof}

\begin{remark} 
The divisor class of $\PP\HH(2, 1^{2g-4})$ was first calculated by Korotkin and Zograf \cite[Theorem 2]{kztau} using the Tau function. After finishing the paper the author learnt from Zograf that another explanation of the divisor class was recently discovered by Zvonkine \cite{zvonkine}. Comparing Theorem~\ref{thm:class} with \cite{kztau}, we see that the corresponding coefficients are opposite to each other. This sign issue is exactly due to different conventions of projectivization parameterizing lines or hyperplanes. Consequently in our setting the classes $\lambda$, $\psi$ and etc are the opposites to those in \cite{kztau}. 
\end{remark}

\begin{example}
Consider the case $g=3$. Take a general pencil $B$ of plane cubics. Let $L$ be a general line in $\PP^2$. The section of every quartic in $B$ with $L$ defines
a canonical divisor. We have $ B\cdot \lambda = 3$, $B\cdot \delta_0 = 27$ and $B\cdot \delta_i = 0$ for $i > 0$, 
see \cite[Chapter 3.F]{harrismorrison}. The universal canonical divisor has class $(1,4)$ in $B\times L \cong \PP^1\times \PP^1$. Projecting it to $B$ induces 
a degree $4$ covering map. If along $L$ a simple ramification occurs, it gives rise to a canonical divisor with a zero of multiplicity two. 
By the Riemann-Hurwitz formula, the number of ramifications is equal to $6$, hence we obtain that $B\cdot \PP\BHH(2,1,1) = 6$. 
The Hodge bundle restricted to $B$ is isomorphic to $\OO(1)^{\oplus 3}$. Hence its projectivization is trivial, but the universal line bundle corresponds to 
$\OO(2)$ due to the twist, see \cite[Appendix A]{lazarsfeldI}. Then we conclude that $B\cdot \psi = 2$. One checks that these intersection numbers satisfy the relation in Theorem~\ref{thm:class}. 
\end{example}

Recall that a divisor class is big if it lies in the interior of the cone of pseudo-effective divisors, see \cite[Chapter 2.2]{lazarsfeldI}. 

\begin{proposition}
\label{prop:rigid}
The divisor class $\PP\BHH(2, 1^{2g-4})$ lies on the boundary of the pseudo-effective cone of $\PP\BHH$. 
\end{proposition}

\begin{proof}
If $\PP\BHH(2, 1^{2g-4})$ is big, we can write it as $N + A$, where $N$ is effective and $A$ is ample. Consider Teichm\"uller curves $\TT$ in $\PP\HH(1^{2g-2})$. 
By \cite[Proposition 3.1]{chenmoeller} we know $\TT$ is disjoint with $\PP\BHH(2, 1^{2g-4})$, hence we have 
$$ \TT \cdot (N+A) = 0. $$
Since $A$ is ample, $\TT \cdot A > 0$. Therefore, $\TT\cdot N < 0$ and consequently $N$ contains $\TT$. However, the union of such $\TT$ is Zariski dense in 
$\PP\HH$, see e.g. \cite[Theorem 1.21]{chencovers}. Then we conclude a contradiction. 
\end{proof}

\section{Extremal effective divisors on $\BM_{g,n}$}
\label{sec:extremal}

We say that an effective divisor class $D$ in a projective variety $X$ is extremal, if for any linear combination $D = D_1 + D_2$ with $D_i$ 
pseudo-effective, $D$ and $D_i$ are proportional. In this case, we also say that $D$ spans an extremal ray of the pseudo-effective cone $\BEff(X)$. Let us first present 
a method to test the extremality of an effective divisor. 

\begin{lemma}
\label{lem:negative}
Suppose that $D$ is an irreducible effective divisor and $A$ an ample divisor in $X$. Let $S$ be a set of irreducible effective curves in $D$ such that the union of these curves is Zariski dense in $D$. If for every curve $C$ in $S$ we have 
$$\frac{C\cdot D}{C\cdot A} \leq - d $$
for fixed $d > 0$, then $D$ is an extremal divisor.   
\end{lemma}

\begin{proof}
Suppose that $D = D_1 +  D_2$ with $D_i$ pseudo-effective. If $D_i$ and $D$ are not proportional, we can assume that 
$D_i$ lies in the boundary of $\BEff(X)$ and moreover that $D_i - s D$ is not pseudo-effective for any $s > 0$, because otherwise we can replace $D_1$ and $D_2$ by the intersections of the linear span $\langle D_1, D_2 \rangle$ with the boundary of $\BEff(X)$, possibly after rescaling. 

By assumption, we have $C\cdot (D_1 + D_2) = C\cdot D < 0$. Therefore, without loss of generality we may assume that $S$ has a subset $S_1$ whose elements $C$ satisfy  
$$C\cdot D_1 \leq \frac{1}{2} \cdot (C\cdot D)$$  
and the union of $C$ in $S_1$ forms a dense subset of $D$ as well. 

Consider the divisor class $F_n = n D_1 + A$ for $n$ sufficiently large. Since $D_1$ is pseudo-effective and $A$ is ample, $F_n$ can be represented by an effective divisor. 
It is easy to check that for $k < \frac{n}{2} -\frac{1}{d}$, we have $C\cdot (F_n - kD) < 0$ for every $C$ in $S_1$. Since such curves $C$ form a dense subset in $D$, it implies that the multiplicity of $D$ in the base locus of $F_n$ is at least equal to $\frac{n}{2} -\frac{1}{d}$. Consequently the class 
$$E_n = F_n - \Big(\frac{n}{2} -\frac{1}{d}\Big) D$$ 
is pseudo-effective. As $n$ goes to infinity, the limit of the divisor classes $\{\frac{1}{n} E_n\}$ is equal to $D_1 - \frac{1}{2} D$, which is also pseudo-effective. But this contradicts our assumption that $D_1 - s D$ is not pseudo-effective for any $s > 0$. 
\end{proof}

In what follows we will apply Lemma~\ref{lem:negative} to Teichm\"uller curves contained in a stratum of abelian differentials. 

Consider the moduli space $\BM_{g,n}$ of stable genus $g$ curves with $n$ ordered marked points. Since Teichm\"uller curves form a Zariski dense subset 
in any (connected component of) stratum $\HH(\mu)$, if $\HH(\mu)$ dominates an irreducible effective divisor $D$ in $\BM_{g,n}$, then the images of 
these Teichm\"uller curves also form a Zariski dense subset in $D$. In order to apply Lemma~\ref{lem:negative} to show the extremality of $D$, we need to understand the intersection of a Teichm\"uller curve with divisor classes on $\BM_{g,n}$. Luckily this has been worked out in \cite[Section 4]{chenmoeller}. For the reader's convenience, in what follows we recall the related results.

Let $C$ be (the closure of) a Teichm\"uller curve in the stratum $\HH(\mu)$, where $\mu = (m_1, \ldots, m_n)$ is a partition of $2g-2$. Let $L$ be the sum of Lyapunov exponents of $C$ and $\chi$ its orbifold Euler characteristic. Lift $C$ to $\BM_{g,n}$ by marking the $n$ zeros of its generating abelian differential. Let $\omega_i$
be the first Chern class of the relative dualizing sheaf associated to forgetting the $i$th marked point. Recall that 
$$\kappa_{\mu} = \frac{1}{12}\sum_{i=1}^n \frac{m_i (m_i + 2)}{m_i+1}. $$
By \cite[Proposition 4.8]{chenmoeller} we have 
$$ C\cdot \lambda = \frac{\chi}{2}\cdot L, $$
$$ C\cdot \delta_0 = \frac{\chi}{2}\cdot (12L - 12\kappa_{\mu}), $$
$$ C\cdot \omega_i = \frac{\chi}{2} \cdot \frac{1}{m_i + 1}. $$

Moreover, we use $\delta_{\other}$ to denote an arbitrary linear combination of boundary divisors of $\BM_{g,n}$ that does not contain $\delta_{0}$. The purpose of doing this is because Teichm\"uller curves generated by abelian differentials do not intersect any boundary divisors except $\delta_0$ \cite[Corollary 3.2]{chenmoeller}. Therefore, we can write a divisor class in $\Pic_{\mathbb Q}(\BM_{g,n})$ as 
$$ D = a\lambda + \sum_{i=1}^n b_i \omega_i + c \delta_0 + \delta_{\other}. $$
By the above intersection numbers, we have 
$$ \frac{C\cdot D}{C\cdot \lambda} = a + \sum_{i=1}^n \frac{b_i}{(m_i+1)L} + c \Big(12 - \frac{12\kappa_\mu}{L}\Big). $$

Let $\underline{a} = (a_1, \ldots, a_n)$ be a sequence of positive integers such that $\sum_{i=1}^n a_i = g$. Consider the pointed Brill-Noether divisor 
$BN^1_{g, \underline{a}}$ in $\BM_{g,n}$ parameterizing $(X, p_1,\ldots, p_n)$ such that 
$h^0(X, \sum_{i=1}^n a_i p_i)\geq 2$. Its divisor class was first calculated in \cite{Logan} as 
$$ BN^1_{g, \underline{a}} = -\lambda + \sum_{i=1}^n \frac{a_i (a_i+1)}{2}\omega_i - \delta_{\other}. $$
We have a dominant map $\HH(\underline{a}, 1^{g-2})\to BN^1_{g,\underline{a}}$ by marking the first $n$ zeros of an abelian differential. 

In order to apply Lemma~\ref{lem:negative}, we need an ample divisor on $\BM_{g,n}$. The class $\lambda$ is semi-ample, so certain perturbation 
$$D_{\underline{s}} = \lambda + s_0\delta_0 + \sum_{i=1}^g s_i\omega_i + \delta_{\other}$$ 
is ample, where $\underline{s} = (s_0, \ldots, s_g)$ with $|s_i|$ as small as we want. 

\begin{lemma}
\label{lem:brill-noether}
Let $C$ be a Teichm\"uller curve in $\HH(\underline{a}, 1^{g-2})$ mapping to $BN^1_{g,\underline{a}}$. Then 
$C\cdot BN^1_{g,\underline{a}} < 0$ if and only if $L > \frac{g}{2}$. Moreover, suppose that for every $C$ we have 
$L \geq \frac{g}{2} + \epsilon$ for a given $\epsilon > 0$, then there exists an ample divisor $D_{\underline{s}}$ as above such that 
$$ \frac{C\cdot  BN^1_{g,\underline{a}} }{C\cdot D_{\underline{s}}} \leq - d $$
for some $d > 0$, where $d$ only depends on $\epsilon$, $g$ and $\underline{a}$. 
\end{lemma}

\begin{proof}
We have 
$$ \frac{C\cdot BN^1_{g,\underline{a}}}{C\cdot \lambda} = -1 + \sum_{i=1}^n \frac{a_i}{2L}. $$
The first part of the lemma follows from the assumption that $\sum_{i=1}^n a_i = g$.

For the other part, we have 
$$ \frac{C\cdot D_{\underline{s}}}{C\cdot \lambda} = 1 + s_0 \Big(12 - \frac{12\kappa_{\mu}}{L}\Big) + \sum_{i=1}^n \frac{s_i}{(a_i+1)L},$$
where $\mu = (a_1,\ldots, a_n, 1^{g-2})$. Moreover, we know   
$$L \leq \frac{3g}{g-1} \kappa_{\mu}, $$
which is a direct consequence of the fact that the slope of $C$ satisfies 
$$s(C) = \frac{C\cdot \delta}{C\cdot \lambda} \leq 8 + \frac{4}{g}, $$
see \cite[Chapter 6.D]{harrismorrison}. Then we conclude that 
$$ \frac{C\cdot BN^1_{g,\underline{a}}}{C\cdot D_{\underline{s}}} \leq \frac{-2\epsilon}{(g+2\epsilon) \Big( 1 + |s| \big(8 + \frac{4}{g} + \frac{2}{g}\sum_{i=1}^n\frac{1}{a_i+1}\big)  \Big) }, $$
where $s = \max\{ |s_0|, \ldots, |s_n| \}$. 
As mentioned above, we can take certain $\underline{s}$ with $|s_i|$ arbitrarily small while making $D_{\underline{s}}$ ample. This implies the existence 
of the desired bound $-d$, which is independent of $C$. 
\end{proof}

For $n=1$, $BN^1_{g,(g)}$ parameterizes a genus $g$ curve with a marked Weierstrass point, hence we also use $W$ to denote the divisor in this case. 
 
\begin{theorem}
\label{thm:n=1}
For $2\leq g \leq 4$ the divisor $W$ is extremal in $\BM_{g,1}$. 
\end{theorem}
 
\begin{proof}
Let $C$ be a Teichm\"uller curve in $\HH(g, 1^{g-2})$. For $g=2$ and $g=3$, by \cite[Corollary 4.3, Section 5.2]{chenmoeller} we know $L$ is equal to $\frac{4}{3} > 1$ and $\frac{7}{4} > \frac{3}{2}$, respectively. Then combining Lemmas~\ref{lem:brill-noether} and~\ref{lem:negative} we conclude that $W$ is extremal in $\BM_{2,1}$ and $\BM_{3,1}$. 

For $g=4$, despite that Teichm\"uller curves in $\HH(4,1,1)$ have varying sums
$L$ of Lyapunov exponents, the limit of $L$ is equal to the sum $L_{(4,1,1)}$ of Lyapunov exponents associated to the whole stratum \cite[Appendix A]{chenrigid}. 
Based on the recursive algorithm in \cite{emz}, we know $L_{(4,1,1)} = \frac{1137}{550} > 2$ \cite[Figure 3]{chenmoeller}. It implies that we can find infinitely many Teichm\"uller curves in $\HH(4,1,1)$ such that they form a Zariski dense subset and all of them have $L > 2+\epsilon$ for some fixed $\epsilon > 0$. Then the result follows by combining Lemmas~\ref{lem:brill-noether} and~\ref{lem:negative}. 
\end{proof}

\begin{remark}
The extremality of $W$ was first showed by Rulla for $g=2$ \cite{rulla} and by Jensen for $g=3$ and $g=5$ \cite{jensen34, jensen56}, using different techniques. Theorem~\ref{thm:n=1} enriches the list by adding the case $g=4$. The question remains open to determine whether $W$ is extremal for general $g$.  
\end{remark}



Next we consider $\BM_{g,2}$. 

\begin{theorem}
\label{thm:n=2}
The divisors $BN^1_{2,(1,1)}$, $BN^1_{3, (2,1)}$, $BN^1_{4, (3,1)}$ and $BN^1_{4,(2,2)}$ are extremal. 
\end{theorem}

\begin{proof}
Teichm\"uller curves in $\HH(1,1)$ and $\HH(2,1,1)$ have $L = \frac{3}{2} > 1$ \cite[Corollary 4.3]{chenmoeller} and $L= \frac{11}{6} > \frac{3}{2}$ 
\cite[Section 5.4]{chenmoeller}, respectively. By Lemmas~\ref{lem:brill-noether} and~\ref{lem:negative} we conclude the extremality for $BN^1_{2,(1,1)}$ and $BN^1_{3, (2,1)}$. 

By \cite[Figure 3]{chenmoeller} the limit of $L$ for Teichm\"uller curves in $\HH(3,1^3)$ and in $\HH(2,2,1,1)$
is equal to $\frac{66}{31} > 2$ and $\frac{5045}{2358} > 2$, respectively. By the same argument as in the proof of Theorem~\ref{thm:n=1} we conclude 
the extremality for $BN^1_{4, (3,1)}$ and $BN^1_{4,(2,2)}$. 
\end{proof}

Now let us consider $\BM_{g,3}$. 

\begin{theorem}
\label{thm:n=3}
The divisors $BN^1_{3,(1^3)}$ and $BN^1_{4,(2,1,1)}$ are extremal. 
\end{theorem}

\begin{proof}
The proof is the same as above by using Teichm\"uller curves in $\HH(1^4)$ whose values of $L$ have limit equal to
$\frac{53}{28} > \frac{3}{2}$ \cite[Figure 2]{chenmoeller} as well as Teichm\"uller curves in $\HH(2,1^4)$ whose values of $L$ have limit equal to
$\frac{131}{60} > 2$ \cite[Figure 3]{chenmoeller}. 
\end{proof}

Then we consider $\BM_{g,4}$. 

\begin{theorem}
\label{thm:n=4}
The divisor $BN^1_{4, (1^4)}$ is extremal. 
\end{theorem}

\begin{proof}
The proof is the same as above by using Teichm\"uller curves in $\HH(1^6)$ whose values of $L$ have limit equal to
$\frac{839}{377} > 2$ \cite[Figure 3]{chenmoeller}. 
\end{proof}

Finally we consider $\BM_{g,5}$. 

\begin{theorem}
\label{thm:n=5}
The divisor $BN^1_{5, (1^5)}$ is extremal. 
\end{theorem}

\begin{proof}
The proof is the same as above by using Teichm\"uller curves in $\HH(1^8)$ whose values of $L$ have limit equal to
$\frac{235761}{93428} > \frac{5}{2}$ \cite[Figure 5]{chenmoeller}. 
\end{proof}

\bibliography{my}

\end{document}